\documentclass[12pt,reqno]{amsart}

\usepackage{latexsym}
\usepackage{amssymb}
\usepackage{graphics}
\usepackage{graphicx}
\usepackage{color}

\renewcommand{\epsilon}{\varepsilon}

\newcommand{\kahler}{K\"ahler }
\newcommand{\image}{image }

\newcommand{\PP}{{\mathbb P}}

\newcommand{\C}{{\mathbb C}}

\newcommand{\Z}{{\mathbb Z}}

\newcommand{\CP}{\C\PP}

\newcommand{\dbar}{\bar\partial}
\newcommand{\ddbar}{\partial\dbar}

\newcommand{\D}{{\mathbf D}}

\renewcommand{\phi}{\varphi}

\newcommand{\ccal}{\mathcal{C}}

\newcommand{\hcal}{\mathcal{H}}

\newcommand{\mcal}{\mathcal{M}}

\newcommand{\ocal}{\mathcal{O}}

\newtheorem{theo}{{Theorem}}[section]

\newtheorem{lem}[theo]{{Lemma}}

\newtheorem{defin}[theo]{{\sc Definition}}

\title[stable degeneration of curves]{Projective embedding of stably degenerating sequence of hyperbolic Riemann surfaces }

\author{Jingzhou Sun }
\address{Department of Mathematics, Shantou University, Shantou City, Guangdong
	515063,China} \email{jzsun@stu.edu.cn}
\thanks{The author is partially supported by NNSF of China no.11701353 and the STU Scientific Research Foundation for Talents no.130/760181.}

\date{\today}

\begin{document}

	\begin{abstract}
		Given a sequence of genus $g\geq 2$ curves converging to a punctured Riemann surface with complete metric of constant Gaussian curvature $-1$.
		we prove that the Kodaira embedding using orthonormal basis of the Bergman space of sections of a pluri-canonical bundle also converges to the embedding of the  limit space together with extra complex projective lines.
	\end{abstract}

	\maketitle
	
	\tableofcontents
	\section{Introduction}
	Let $\mcal_g$ be the moduli of smooth compact Riemann surfaces of genus $g$. When $g\geq 2$, the Deligne-Mumford compactification \cite{DeligneMumford}  $\overline{\mcal_g}$ is the moduli of stable curves.  Each smooth curve of genus $g$ carries an unique Poincar\'{e} metric with constant Gaussian curvature $-1$. If $C\in \overline{\mcal_g}$ is a singular stable curve, then by removing the nodes, the smooth part carries an unique complete hyperbolic metric with constant Gaussian curvature $-1$.
	 And if a holomorphic family $\pi: \ccal \to \D $ of compact smooth curves $C_t$ degenerate to $C=C_0$, then the hyperbolic metrics is continuous on the vertical line bundle  \cite{WolpertHyperbolic}.
	 
	 \
	 
	 In this article, from the point view of the quantization framework by Donaldson \cite{donaldson2001, Donaldson2014Gromov}, we are interested in the convergence of the pluri-canonical Bergman embeddings of the family of hyperbolic surfaces in the complex projective spaces. 
	More precisely, let $(C_j,g_j)$ be a sequence of genus $g\geq 2$ Riemann surfaces with Riemannian metric $g_j$ of constant Gaussian curvature $-1$, that converges, in the topology of pointed Gromov-Hausdorff, to a Punctured Riemann surface $(C_0,g_0)$(not necessarily connected) with a complete Riemannian metric $g_0$ of constant Gaussian curvature $-1$. Let $K_{C_j}$ denote the canonical bundle of $C_j$, then $K_{C_j}$ is endowed with a Hermitian metric $h_j$ defined by the \kahler form $\omega_j$ associated to $g_j$. We consider the Bergman space $\hcal_{j,k}$ consisting of $L^2$-integrable holomorphic sections of $K_{C_j}^k$. Then $\hcal_{j,k}$ is a finite-dimensional Hermitian space with the Hermitian product defined by
	$$<s,t>=\int_{C_j}(s,t)_{h_j}\omega_j, $$
	where, by abuse of notation, we still use $h_j$ to denote the induced Hermitian metric on $K_{C_j}^k$. For $k$ large enough, a basis of $\hcal_{j,k}$ will induce a Kodaira embedding of $C_j$ to $\CP^{N_k}, $ where $N_k=\dim \hcal_{j,k}-1$ is independent of $j\geq 1$. For $j=0$, the dimension of $\hcal_{j,k}$ is smaller than that of $j>0$. It is natural to consider the embedding induced by an orthonormal basis for $\hcal_{j,k}$, which can be considered as a bridge from \kahler geometry to complex geometry \cite{DonaldsonSun2, SunSun}. It is worth mentioning that after this article is finished, the author learned that Dong\cite{RobertDong} recently proved that if a smooth family of hyperelliptic curves degenerate to a nodal curve, then their Bergman kernels also converges to the Bergman kernel of the nodal curve.

	As the Gaussian curvature is $-1$, the degeneration of metrics can only be "pinching a nontrivial loop", namely a sequence of surfaces with growingly thiner and longer handles, with the central loops degenerating to points. So $C_0$ has $d$ pairs of punctures, which will be called ends. And for $k$ large enough,the dimension of $\hcal_{0,k}$ equals $N_k+1-d$. Now we can state our main theorem.
	\begin{theo}
		For $k$ large enough, we can choose an orthonormal basis for $\hcal_{j,k}$ for all $j>0$, so that as $j\to \infty$ the image of the embedding
		$$\Phi_{j,k}:C_j\to \CP^{N_k}$$
		 induced by the orthonormal basis converges to the image of $C_0$ under the embedding
		 $$\Phi_{0,k}:C_0\to \CP^{N_k-d}\subset \CP^{N_k}, $$ attached with $d$ pairs of linear $\CP^1$'s. To each pair of the ends $(p_\alpha,p_{\alpha+d})$, a pair of linear $\CP^1$'s are associated, and form a connected chain connecting the images of these two points.
	\end{theo}
It is interesting to mention that during the process of taking limit, the pair of linear $\CP^1$'s are developed as a pair of bubbles.
Also, we should mention that $k$ depends only on the geometry of $C_0$, and does not need to be too big by the results in \cite{SunPunctured}. 

The proof of this theorem makes heavy use of the methods we developed from \cite{SunSun} to \cite{SunPunctured}. And just as in  \cite{Donaldson2014Gromov}, the main point is basically proving the convergence of the Bergman kernels. And we hope this result may shine a light on the study of the degeneration of higher dimensional projective manifolds \cite{SunZhang,JianSong2017, Annalen2019}.

\

	The structure of this article is as follows. We will first quickly recall the necessary background for this article. Then we will calculate in the model for the thin handles, or "the collar", of the Riemann surfaces close to the limit. And in the end, we will finish the proof of the convergence of the pluri-canonical Bergman embeddings.
	
	\
	
	\textbf{Acknowledgements.} The author would like to thank Professor Song Sun for many very helpful discussions.

	\

 \section{Punctured Model}

The model $\D^*$ with the Poincar\'{e} metric
$$\omega_{P}=\frac{2idz\wedge d\bar{z}}{|z|^2(\log |z|^2)^2}$$
Take the local section of the canonical bundle
$e=\frac{dz}{z}$, the local potential is 
$$\phi_P=-\log |e|^2=-\log \frac{(\log \frac{1}{|z|^2})^2}{2}$$
We use the notation $\tau=-\log |z|$, so $\phi_P=-\log(2\tau^2)$.
We are interested in the $L^2$-norm of the sections $z^ae^{k+1}$ of $K_{\D^*}^{k+1}$, $a\in \Z^+$. So we have the following integrals 
$$Y_a=\int (2\tau^2)^{k+1}|z|^{2a}\omega_P$$
We have 
$$Y_a=2^{k+2}\pi\int_{0}^{\infty} e^{-2a\tau+2k\log \tau}d\tau$$
We denote by $g_a(\tau)=-2a\tau+2k\log \tau$, then $g_a''(t)=-\frac{2k}{\tau^2}$. So $g_a(\tau)$ is a concave function which attains its only maximum at $\tau_a=\frac{k}{a}$. We will use the following basic lemma from \cite{SunPunctured}.
\begin{lem}\label{lemconcave}
	Let $f(x)$ be a concave function. Suppose $f'(x_0)<0$, then we have
	$$\int_{x_0}^\infty e^{f(x)}dx\leq\frac{e^{f(x_0)}}{-f'(x_0)}$$
\end{lem}	 We can use Laplace's method and the lemma above to estimate
$$Y_a\approx 2^{k+2}\pi e^{-2k+2k\log \frac{k}{a}}\sqrt{\frac{k\pi}{2a^2}}$$
Of course, we can directly calculate the integral to get
$$Y_a=2^{k+2}\pi\frac{(2k)!}{(2a)^{2k+1}},$$
But the idea of mass concentration is key to our arguments. 
The Bergman kernel of $\D^*$ is then 
$$\rho_{0,k+1}= \frac{2^{2k}\tau^{2k+2}}{\pi (2k)!}\sum (a)^{2k+1}|z|^{2a}$$
Let $C_0$ be a punctured Riemann surface obtained by removing $2d$ points $\{p_\alpha\}_{1\leq \alpha\leq 2d}$ from a compact Riemann surface. $C_0$ is endowed with a complete Poincar\'{e} metric $\omega$ with constant Gaussian curvature $-1$. $\omega$ defines a Hermitian metric $h$ on the canonical bundle $K_{C_0}$. Then, for any positive integer $k$, we denote by $\hcal_k$ the space of holomorphic sections of $K_{C_0}^k$ that are $L^2$-integrable, namely
$$\int_{C_0}|s|_h^2\omega<\infty. $$ 
For each $p_\alpha$, there is a neighborhood $U_\alpha$ with local coordinate $z$ so that $\omega=\omega_P$ on $U_\alpha\backslash p_\alpha$. We can assume that $U_\alpha$ contains the points satisfying $|z|\leq R_\alpha$. We note that the injective radius at the points $|z|= R_\alpha$ is about $\frac{\pi}{4(\log R_\alpha)^2}$. For simplicity, we let $R$ be the minimum of the $R_\alpha$'s, $1\leq \alpha\leq 2d$. 
Clearly, for the complement of $\cup_{1\leq \alpha\leq 2d}U_\alpha$ in $C_0$, there is a positive lower bound $\lambda_0$ for the injective radius. 
The basic conclusion of \cite{SunPunctured} is that for $k$ large enough, in the "inside" of $U_\alpha$ where $\tau=-\log |z|>\sqrt{k+1}$, the Bergman kernel $\rho_{k+1}$ is very much like $\rho_{0,k+1}$, which is dominated by the terms $c_a|z|^{2a}$ for $a<k^{3/4}$. In particular, when $\tau\geq k$, $\rho_{0,k+1}$ is dominated by $c_1|z|^2$. The sections for $C_0$ corresponding to $z^a$ in the model $\D^*$ is constructed as follows.
 We let $z_\alpha$ denote the local coordinate $z$ on $U_\alpha$. Let $e_\alpha=\frac{dz_\alpha}{\alpha}$ be the local frame of $K_{C_0}$. Then $z_\alpha^ae_\alpha^{k+1}$, $a\geq 1$, are local sections of $K_{C_0}^{k+1}$. We choose and fix a cut-off function $\chi(r)$ that equals $1$ for $r<R/2$ and that equals 0 for $r>2R/3$. Then we denote by $\chi_\alpha$ the function $\chi(|z_\alpha|)$ defined on $C_0$. Then $\chi_\alpha z_\alpha^ae_\alpha^{k+1}$ is a global smooth $L^2$-integrable section of $K_{C_0}^{k+1}$. We then take the orthogonal projection of this section into the space $\hcal_{k+1}$, and then normalize the holomorphic section to be of norm 1, obtaining a section $s_{\alpha,a}\in \hcal_{k+1}$. We denote by 
$$V_0=\{s_{\alpha,a},\quad 1\leq \alpha\leq 2d,\quad 1\leq a<k^{3/4} \}$$
For $k$ large enough, within $r<R/4$, the sections $s_{\alpha,a}\approx \sqrt{c_a}z^a$ with relative error less than $\frac{1}{k^2}$.

We choose and fix an orthonormal basis $W_0=\{s_j\}$ for the orthogonal complement $V_0^\perp\subset \hcal_{k+1}$. 

\

To obtain global sections of $L^k$ from local ones, we will need to use H\"ormander's $L^2$ estimate. The following lemma is well-known, see for example \cite{Tian1990On}. 

\begin{lem}\label{lemHorm}
	Suppose $(M,g)$ is a complete \kahler manifold of complex dimension $n$, $\mathcal L$ is a line bundle on $M$ with hermitian metric $h$. If 
	$$\langle-2\pi i \Theta_h+Ric(g),v\wedge \bar{v}\rangle_g\geq C|v|^2_g$$
	for any tangent vector $v$ of type $(1,0)$ at any point of $M$, where $C>0$ is a constant and $\Theta_h$ is the curvature form of $h$. Then for any smooth $\mathcal L$-valued $(0,1)$-form $\alpha$ on $M$ with $\bar{\partial}\alpha=0$ and $\int_M|\alpha|^2dV_g$ finite, there exists a smooth $\mathcal L$-valued function $\beta$ on $M$ such that $\bar{\partial}\beta=\alpha$ and $$\int_M |\beta|^2dV_g\leq \frac{1}{C}\int|\alpha|^2dV_g$$
	where $dV_g$ is the volume form of $g$ and the norms are induced by $h$ and $g$.
\end{lem}
In the setting of this article, for a curve $C_j$, $j\geq 0$, with line bundle $K_{C_j}^{k+1}$, the constant is $k$, independent of $j$. 
\section{the Collar model}

The model $\C^*_{\epsilon}=(\C^*,\omega_{\epsilon})$, where 
$$\omega_{\epsilon}=\frac{f_{\epsilon}idz\wedge d\bar{z}}{2|z|^2},$$
where $f_\epsilon$ is a function depending only on $|z|$, satisfying the following conditions
\begin{itemize}
	\item[$\bullet$]  $f_\epsilon>0$;
	\item[$\bullet$]  $f_\epsilon(1)=\epsilon^2$;
	\item[$\bullet$]  $f'_\epsilon(1)=0$;
	\item[$\bullet$]  $\Delta \log f_\epsilon=\frac{2f_\epsilon}{|z|^2}$;
\end{itemize}
Clearly, such $f_\epsilon$ exists and is unique. Also, our choice of $f_\epsilon$ makes the metric have constant Gaussian Curvature $-1$. We denote by $t=\log |z|$, then by abuse of the notation, we consider $f_\epsilon$ as a function of $t$. Then we have 
$$\Delta_z \log f_\epsilon=\frac{d^2 \log f_\epsilon}{dt^2}\frac{1}{|z|^2}$$
For simplicity, we will use $f(t)$ to denote $f_\epsilon$.
Therefore, we have
$$(\log f(t))''=2f(t)$$
The first fundamental form of the metric is 
$$I=f(t)dt^2+f(t)d\theta^2$$
Clearly, by the requirements on $f$, the circle $\{z||z|=1\}$ is a geodesic.
Then we use the arc-length parameter $u$ for the $t$-curves. Then by the curvature condition, we have 
$$f(t)=\epsilon^2\cosh^2 u, \quad u(0)=0$$
and $\frac{dt}{du}=\frac{1}{\epsilon\cosh u}$. Therefore, we have 
$$t=\frac{2}{\epsilon}\tan^{-1}[\tanh\frac{u}{2}]$$
So we have
$$\frac{\epsilon t}{2}=\tan^{-1}\sqrt{1-\frac{2}{\cosh u+1}}$$
Therefore when $\cosh u$ is large, we have the following estimation
\begin{equation}\label{formula1}
t=\frac{\pi}{2\epsilon}-\frac{1}{\epsilon(\cosh u+1)}+O(\frac{\epsilon}{(\epsilon(\cosh u+1))^2})
\end{equation}
So  it is natural to use the notations $$\tau_\epsilon=\frac{\pi}{2\epsilon}-t.$$ 
We also use the frame $e=\frac{dz}{z}$ for the canonical bundle, so we have
$$|e|^2=\frac{2}{f}$$

 \

We are interested in the $L^2$-norm of the sections $z^ae^{k+1}$ of $K_{\C^*}^{k+1}$, $a\in \Z$. So we have the following integrals 
$$I_{\epsilon,a}=\int (\frac{2}{f})^{k+1}|z|^{2a}\omega_{\epsilon}$$
We have 
$$I_{\epsilon,a}=2^{k+2}\pi\int_{-\infty}^{\infty} e^{2at-k\log f}dt$$
We denote by $g_a(t)=2at-k\log f$, then $g_a''(t)=-2kf(t)$. So $g_a(t)$ is a concave function which attains its only maximum at $t_a$. Write $u_a=u(t_a)$, we have 
$$\epsilon\sinh u_a=\frac{a}{k}$$
We will assume that $\epsilon$ is very small compared to $k^{-k}$. So $\sinh u_a=\frac{a}{k\epsilon}$ is very large. So $f(t_a)>\frac{a^2}{k^2}$, and we can use Laplace's method and lemma \ref{lemconcave} to estimate

$$I_{\epsilon,a}\approx 2^{k+2}\pi e^{2at_a-k\log f(t_a)}\sqrt{\frac{\pi}{2kf(t_a)}}$$
And we have that the mass of $I_{\epsilon,a}$ is concentrated within the neiborhood $\{|t-t_a|<\frac{\sqrt{k}\log k}{a}\}$ with relative error less than $k^{-\log k+3/2}$. Also, when $\epsilon$ is small,
$$\frac{I_{\epsilon,a+1}}{I_{\epsilon,a}}\approx e^{\pi/\epsilon}\frac{a^{2k+1}}{(a+1)^{2k+1}}$$
Therefore, the power series $\sum_{a>0}\frac{|z|^{2a}}{I_{\epsilon,a}}$ is very close to a multiple of the power series
$$\sum_{a>0}a^{2k+1}e^{-2a\tau_\epsilon}.$$
 Recall that the power series in the expression of $\rho_{0,k+1}$ is also
$$\sum a^{2k+1}|z|^{2a}=\sum a^{2k+1}e^{-2a\tau}$$
So by the same argument in \cite{SunPunctured}, for $t\in [0,t_1]$ the Bergman kernel is dominated by $[\frac{|z|^2}{I_{\epsilon,1}}+\frac{1}{I_{\epsilon,0}}](\frac{2}{f})^{k+1}$, and, by symmetry, for $t\in [t_{-1},0]$ the Bergman kernel is dominated by $[\frac{|z|^{-2}}{I_{\epsilon,-1}}+\frac{1}{I_{\epsilon,0}}](\frac{2}{f})^{k+1}$. In particular, we have the following
\begin{lem}\label{lem-cut-tail}
	For any holomorphic section $s$ of $K^{k+1}_{\C^*}$, satisfying $\parallel s\parallel=1$, we have
	$$|s|^2<\epsilon(\frac{\log \epsilon}{k})^{2k}$$
	when $\cosh u\in (\frac{-1}{2\epsilon\log \epsilon},\frac{-1}{\epsilon\log \epsilon})$. 
\end{lem}
\begin{proof}
	By symmetry, we can assume $t>0$. For the right end of the interval, we only need to estimate the norms of $\frac{z}{I_{\epsilon,1}}e^{k+1}$ and  $\frac{1}{I_{\epsilon,0}}e^{k+1}$ at $t$ where $\cosh u=\frac{-1}{\epsilon\log \epsilon}$. For the first one, we have
	$$|\frac{z}{I_{\epsilon,1}}e^{k+1}|^2\approx \frac{\sqrt{2k}\epsilon}{4k\pi^{3/2}}(\frac{\log \epsilon}{k})^{2k}$$
	For the second one, we have 
	$$|\frac{1}{I_{\epsilon,0}}e^{k+1}|^2\approx \frac{\sqrt{2k}}{4\pi^{3/2}}\epsilon^{2k+1}(\log \epsilon)^{2k},$$
	which is much smaller than the first one. For the left end of the interval, we have smaller norm for the section $\frac{z}{I_{\epsilon,1}}e^{k+1}$, and still very small norm for the section $\frac{1}{I_{\epsilon,0}}e^{k+1}$. Combining these estimates, we have proved the lemma.
\end{proof}

Assume $C_j$ converges to $C_0$ in the pointed Gromov-Hausdorff topology. For $j$ big enough, $C_j$ has exactly $d$ closed geodesics whose arc length is less than $\lambda_0/4$. We denote these circles by $\gamma_{j,\alpha}$, $1\leq \alpha\leq d$, and arc length of $\gamma_{j,\alpha}$ by $\epsilon_{j,\alpha}$. By rearranging the points $p_{\alpha}$, we can assume that $2\pi \epsilon_{j,\alpha}$ converges to the pair $(p_\alpha,p_{\alpha+d})$ as $j\to \infty$. Also for $j$ large enough, there a neighborhood $U_{j,\alpha}$, usually referred to as a collar , of each $\gamma_{j,\alpha}$ which is homeomorphic to an annulus. We define a map 
$$h_{j,\alpha}:U_{j,\alpha}\to \C^*_\epsilon$$
with $\epsilon=\epsilon_{j,\alpha}$, as following. Fix an isometry $\lambda$ of $\gamma_{j,\alpha}$ to the circle $|z|=1$ in $\C^*_\epsilon$. Then passing through each point $q$ on $\gamma_{j,\alpha}$, there is an unique geodesic $l_q$ orthogonal to $\gamma_{j,\alpha}$. Then we define $h_{j,\alpha}$ to be the map that sends each such geodesic $l_q$ to the geodesic passing through $\lambda(q)$ and being orthogonal to the unit circle, preserving $\lambda$ and the orientation. Since both surfaces have constant Gaussian curvature $-1$, $h_{j,\alpha}$ is an isometry so long as the geodesics $l_q$ do not intersect each other. But since the curvature is negative, by the Gauss-Bonnet theorem, these geodesics can not intersect within $U_{j,\alpha}$. Therefore, $h_{j,\alpha}$ is also holomorphic. so we can use the coordinate $z$ from $\C^*_\epsilon$ as the holomorphic coordinate of $U_{j,\alpha}$. By switching $p_\alpha$ and $p_{\alpha+d}$ if necessary, we can assume that the part $|z|>1$ of $U_{j,\alpha}$ converges to a neighborhood of $p_\alpha$ and the part $|z|<1$ to that of $p_{\alpha+d}$. We can assume that $U_{j,\alpha}=\{1/M\leq |z|\leq M\}$ and for $j$ large enough, we can assume that the injective radius at $|z|=M$ is larger than $\frac{\pi}{4(\log 3R/4)^2}$. We denote by $U_{j,\alpha}^+$ the part of $U_{j,\alpha}$ with $|z|>1$, similarly $U_{j,\alpha}^-$ the part with $|z|<1$. We then define a map 
$$\phi_{j,\alpha}:U_{j,\alpha}^+\to U_\alpha$$
by sending $\epsilon\cosh u$ to $\frac{1}{2\tau}$ while preserving the circles $\{u=\textit{constant}\}$. Clearly, we are only preserving the length of the circles. By symmetry, we also have
$$\phi_{j,\alpha+d}:U_{j,\alpha}^-\to U_{\alpha+d}.$$
By our assumption on the injective radius, the image of $\phi_{j,\alpha}$ contains the circle $|z_\alpha|=\frac{3R}{4}$.
On $U_\alpha$, the first fundamental form is
$$I_0=\frac{1}{\tau^2}(d\tau^2+d\theta^2)$$
So the pull back
$$\phi_{j,\alpha}^*I_0=\tanh^2u du^2+(\epsilon\cosh u)^2d\theta^2$$
is almost isometric to the metric
$$I_j=du^2+(\epsilon\cosh u)^2d\theta^2,$$ when $u$ is large. In particular, for the part where $\epsilon\cosh u\geq \frac{-1}{\log \epsilon}$, $\phi_{j,\alpha}$ converges to an isometry when $j\to \infty$.

\

Let $U_\alpha(r)$ denote the subset of $U_\alpha$ consists of the points $|z_\alpha|<r$. Let $F=C_0\backslash \cup_{1\leq \alpha\leq 2d}U_\alpha(\frac{2R}{3})$, and let $\psi_j:F\to C_j$ be the diffeomorphism with its image. Since $\psi_j$ converges to an isometry as $j\to \infty$, we can glue $\psi_j^{-1}$ with the $\phi_{j,\alpha}$'s, by rotating $\phi_{j,\alpha}$ if necessary, for $j$ large enough, to get a map
$$G_j:C_j\backslash \cup \gamma_{j,\alpha}\to C_0,$$
with the following properties:
\begin{itemize}
		\item[$\bullet$] $G_j$ is a diffeomorphism of $C_j\backslash \cup \gamma_{j,\alpha}$ with its image.
	\item[$\bullet$] $G_j=\phi_{j,\alpha}$ for $p\in \phi_{j,\alpha}^{-1}U_\alpha(\frac{2R}{3})$, $1\leq \alpha\leq 2d$. 
		\item[$\bullet$] $G_j$ is almost an isometry on $C_j\backslash\cup_{1\leq \alpha\leq 2d} \phi_{j,\alpha}^{-1}U_\alpha(\frac{2R}{3})$, and converges to an isometry when $j\to \infty$.
\end{itemize} 

For any conformal metric, the compatible complex structure $J$ is just a counterclockwise rotation by $\frac{\pi}{2}$. We see that almost isometry implies almost holomorphic. Therefore $G_j^{-1*}K_{C_j}$ converges to $K_{C_0}$ as subbundles of $T_{C_0}\otimes \C$. More precisely, let $J_j$ be the complex structure compatible with the Riemannian metric $g_j$,  if the point-wise norm 
$$\sup_{v\in T_p,|v|_g=1}|g_j(v,v)-g(v,v)|<\delta,$$
then we have
$$\sup_{v\in T_p,|v|_g=1}|J_j(v)-J(v)|_g<\lambda\delta$$
for some constant $\lambda$ independent of $p$ and $g$. We call the supremum above the pointwise distance from $J_j$ to $J$. Moreover, if $g_j$ converges to $g$ in $C^2$-norm, then $J_j$ converges to $J$ in $C^2$-norm also. If we denote by $T_J$ the holomorphic tangent space with respect to $J$, then the orthogonal projection of $T_{J_j}$ to $T_J$ is close to an isometry if $J_j$ is close to $J$. We identify $T_{J_j}$ with $T_j$ via this orthogonal projection, similarly $K_j=T_{J_j}^*$ with $K=T_J^*$, which we will also call an orthogonal projection, for simplicity. Since the metric on the canonical bundle is defined by the \kahler form $\omega$, and $\omega_j$ converges to $\omega$, we have that the Chern connection $\nabla_j$ on $K_j$ converges to the Chern connection $\nabla$ on $K$. 
\
\section{convergence of projective embedding}

By assigning value 1 on $\gamma_{j,\alpha}$, we glue together the pull back $G_j^*\chi_\alpha$ and $G_j^*\chi_{\alpha+d}$ to get a function denoted by $\tilde{\chi}_\alpha$, for $1\leq \alpha\leq d$. 
On each $\phi_{j,\alpha}$, we also consider $\tilde{\chi}_\alpha z^a$ as global smooth sections of $K^{k+1}_{C_j}$. Then we repeat the construction of $V_0$ by normalizing the orthogonal projection of 
$\tilde{\chi}_\alpha z^a$ onto $\hcal_{j,k+1}$, and denote the result section by $s_{j,\alpha,a}$, $|a|<k^{3/4}$. Then we denote by
$$V_j=\{s_{j,\alpha,a},\quad 1\leq \alpha\leq d,\quad |a|<k^{3/4} \}$$
We should remark here that the choice of the upper bound $k^{3/4}$ is not necessary, it is purely a habit from \cite{SunSun}.
Notice that the number of sections of $V_j$ is larger than that of $V_0$ by the number $d$. Those extra sections are $\{s_{j,\alpha,0}\}_{1\leq \alpha\leq d}$. We consider $s_{j,\alpha,a}$ as a smooth section of $G_j^{-1*}K_{C_j}^{k+1}$ on $\image(G_j)$. We then define a piecewise smooth section $\tilde{s}_{j,\alpha,a}$ of $K_{C_0}^{k+1}$ which equals the orthogonal projection of  $s_{j,\alpha,a}$ to $K_{C_0}^{k+1}$ on $\image(G_j)$, and equals 0 in the complement. For simplicity, we will say that $s_{j,\alpha,a}$ converges in some topology if $\tilde{s}_{j,\alpha,a}$ converges in that topology.

	\begin{lem}
		$s_{j,\alpha,a}$ converges to $s_{\alpha,a}$ for $a>0$, to $s_{\alpha+d,-a}$ for $a<0$, in $L^2$-norm, as $j\to \infty$.
	\end{lem}
\begin{proof}
	By symmetry, we only have to prove for $a>0$.
	By taking $j$ large enough, we can assume that $p\in C_0\backslash \cup U_{\alpha}(\epsilon(j,\alpha)) $. Since when $\epsilon\cosh u\geq \frac{-1}{\log \epsilon}$, $\tanh^2u=1-(\epsilon\log \epsilon)^2$ is very close to 1. So $\phi_{j,\alpha}$ is very close to an isometry.  
	For simplicity, we still use $\epsilon$ for short for $\epsilon(j,\alpha)$.	
	Then we look at the integral
	\begin{eqnarray*}
		I_{j,\alpha,a}&=&2^{k+2}\pi\int \tilde{\chi}_\alpha^2 \frac{1}{(\epsilon\cosh u)^{2k}}e^{2at}dt\\
		&=&2^{k+2}\pi e^{\pi a/\epsilon}\int \tilde{\chi}_\alpha^2 \frac{1}{(\epsilon\cosh u)^{2k}}e^{-2a\tau_\epsilon}d\tau_\epsilon
	\end{eqnarray*} $$ $$

On $C_0$, we have
	$$J_{\alpha,a}=2^{k+2}\pi\int \chi_\alpha^2 \tau^{2k} e^{-2a\tau}d\tau $$
	For $I_{j,\alpha,a}$, by lemma \ref{lemconcave}, we can truncate the part $\tau_\epsilon>-\log \epsilon$ by introducing a relative error $<\epsilon$. Also for $J_{\alpha,a}$, we can truncate the part $\tau>-\log \epsilon$ by introducing a relative error $<\epsilon$.
	Then for the part $\tau_\epsilon\leq-\log \epsilon$, 
	$\tilde{\chi}_\alpha^2 \frac{1}{(\epsilon\cosh u)^{2k}}e^{-2a\tau_\epsilon}$ converges to $\chi_\alpha^2 \tau^{2k} e^{-2a\tau}$ uniformly. Therefore
	$I_{j,\alpha,a}e^{-\pi a/\epsilon}$ converges to $J_{\alpha,a}$ as $j\to \infty$. Therefore, $I_{j,\alpha,a}^{-1/2}z_\alpha^a$ converges to $J_{\alpha,a}^{-1/2}z^a$ as $j\to \infty$. Also for this part, the 1-form $\frac{dz_\alpha}{z_\alpha}$ converges uniformly to $\frac{dz}{z}$. The way to get orthogonal projection onto holomorphic sections is to find the solutions of 
	$$\dbar v=\dbar J_{\alpha,a}^{-1/2}z^a(\frac{dz}{z})^{k+1}$$
	and
		$$\dbar_j v_j=\dbar_j I_{j,\alpha,a}^{-1/2}z_\alpha^a(\frac{dz_\alpha}{z_\alpha})^{k+1}$$
	with minimal $L^2$-norms, where we denote by $\dbar_j$ the $\dbar$ operator on $C_j$. And in order to prove the conclusion of the lemma, it suffices to prove that $v_j$ converges to $v$. Notice that $\dbar v$ is supported within the annulus $\frac{2R}{3}\leq |z|\leq \frac{3R}{4}\subset U_\alpha$. And by the mass concentration property of $z^a$, the mass $$\parallel \dbar v\parallel<\frac{1}{k^2}.$$
	Therefore 
	$$\int |v|^2\omega<\frac{1}{k^3},$$
	and $v$ is holomorphic outside the support of $\dbar v$. Since the Bergman kernel is dominated by $\sqrt{Y_1^{-1}}z$, and 
		$$\int_{|z|<\epsilon} Y_1^{-1}|z|^2\omega<\epsilon,$$ 
	 we have 
	$$\int_{|z|<\epsilon} |v|^2\omega<\epsilon$$
	in every $U_\alpha$. We pull back the restriction $v$ to $C_0\backslash \cup U_{\alpha}(\epsilon(j,\alpha)) $ by $G_j$, and use cut-off functions $\kappa_\epsilon$ near the edges to get a global smooth section of  $G_j^*(K_{C_0}^{k+1})$, which is then projected to a smooth section of $K_{C_j}^{k+1}$, called $\tilde{v}_j$. More explicitly, we define the $\kappa_\epsilon$ as a smooth function of $\tau_\epsilon$ that equals 1 for $\tau_\epsilon<-\log \epsilon$ and equals $0$ for $\tau_\epsilon>-2\log \epsilon $. We can also assume that $|\kappa_\epsilon'|<\frac{2}{-\log \epsilon}$.  We have 
	$$\nabla^*\nabla v=2\dbar^*\dbar v+(k+1)v$$
So $$\int |\nabla v|^2\omega=\int 2|\dbar v|^2\omega+(k+1)\int |v|^2\omega <\frac{3}{k^2}. $$
Therefore, $$\int |\nabla_j \tilde{v}_j|^2\omega_j<\frac{4}{k^2}$$
$$\int_{C_j}|\dbar_j \tilde{v}_j-\dbar_jv_j|^2\omega_j=\delta_j,$$
where $\delta_j\to 0$ as $j\to \infty$. Then we can solve the equation $$\dbar_j u=\dbar_j \tilde{v}_j-\dbar_jv_j$$ 
with minimal $L^2$-norm, so that
$$\int_{C_j}|u|^2\omega_j\leq \frac{1}{k}\delta_j$$
Since $v_j$ is a minimal solution, $$\int |\tilde{v}_j-u|^2\omega_j\geq \int |v_j|^2\omega_j.$$
So  $$\int |\tilde{v}_j|^2\omega_j\geq \int |v_j|^2\omega_j-\sqrt{\frac{\delta_j}{k^5}}.$$
Conversely, for each $v_j$, we first use the cut off functions $\kappa_\epsilon$ to make it vanish near the edges, then we pull it back by $G_j^{-1}$ to $C_0$. By orthogonal projection and extension by 0, we obtain a smooth section $\tilde{v}^j$ of $K_{C_0}^{k+1}$. Similar to $\tilde{v}_j$, by lemma \ref{lem-cut-tail}, we have  
$$\int_{C_0}|\dbar \tilde{v}^j-\dbar v|^2\omega_j=\delta_j',$$
where $\delta_j'\to 0$ as $j\to \infty$. Then by solve a $\dbar$-equation again, we get that
$$\int |\tilde{v}^j|^2\omega\geq \int |v|^2\omega-\sqrt{\frac{\delta_j'}{k^5}}.$$
Since the $L^2$ norm of $\tilde{v}_j$ is close to that of $v$, and $$\int |\tilde{v}^j|^2\omega\leq \int |v_j|^2\omega_j+\delta_j'',$$ 
where $\delta_j'\to 0$ as $j\to \infty$, we get that 
$$\int |\tilde{v}_j|^2\omega_j- \int |v_j|^2\omega_j\to 0,$$
 as $j\to \infty$. Then by the uniqueness of the minimal solution of the $\dbar$-equation, we get that $$\int |\tilde{v}_j-v_j|^2\omega_j\to 0$$
 as $j\to \infty$. Finally, since the $L^2$-norm of $s_{j,\alpha,a}$ on the area where $\frac{1}{\epsilon \cosh u}<\frac{-1}{2\log \epsilon}$ also goes to 0 as $j\to \infty$, we have proved the theorem.
	
\end{proof}
	The same ideas can be used for the sections in $W_0$. For each $s\in W_0$, we have
		$$\int_{|z|<\epsilon} |s|^2\omega<\epsilon$$
	in every $U_\alpha$. We pull back the restriction $s$ to $C_0\backslash \cup U_{\alpha}(\epsilon(j,\alpha)) $ by $G_j$, and use cut-off functions $\kappa_\epsilon$ near the edges to get a global smooth section of  $G_j^*(K_{C_0}^{k+1})$, which is then projected to smooth section of $K_{C_j}^{k+1}$, called $\tilde{s}_j$. Then we have
	$$\int_{C_j}|\dbar_j \tilde{s}_j|^2\omega_j\to 0$$
	as $j\to \infty$. So we can solve the equation
	$$\dbar_j u_j=\dbar_j \tilde{s}_j$$
	with minimal $L^2$-norm to get a holomorphic section
	$s_{,j}=u_j+\tilde{s}_j\in \hcal_{j,k+1}$ satisfying
	$\int_{C_j}|s_{,j}|^2\omega_j-1\to 0$ and $s_{,j}\to s$ in $L^2$-norm, as $j\to \infty$. So we have proved the following:

\begin{lem}
	For each $s\in W_0$, we can find $s_{,j}\in \hcal_{j,k+1}$, so that $s_{,j}$ converges to $s$ in the $L^2$-norm, and $\parallel s_{,j}\parallel\to 1$ as $j\to \infty$.
\end{lem}	
	 We will denote by $W_j$ the set of sections $\{s_{,j}|s\in W_0\}$. Notice that $W_j$ is not an orthonormal set, but gets closer as $j$ gets larger. We fix an order on $W_0$, the order each $W_j$ accordingly. We give each $V_j$ the dictionary order. Recall that $s_{j,\alpha,a}$ corresponds to $s_{\alpha,a}$ for $a>0$, $s_{\alpha+d,-a}$ for $a<0$. Then we add to $V_0$ $d$ $0$'s, and order the obtained set $\tilde{V}_0$ according to the correspondence to $V_j$, where each $0$ corresponds to a section $s_{j,\alpha,0}$, $1\leq \alpha\leq d$. Then we define the embedding
	 $$\Phi_j:C_j\to \CP^{N_k},$$
	 where $N_k=\dim \hcal_{j,k+1}-1$, by $\Phi_j=[V_j,W_j]$, where $[\cdots]$ means the homogeneous coordinates. Similarly we define the embedding
 $$\Phi_0:C_0\to \CP^{N_k},$$ by $\Phi_0=[\tilde{V}_0,W_0]$.

Let $[Z_0,\cdots,Z_N]$ be the homogeneous coordinates of $\CP^N$. $U_0=\{[1,w],w\in \C^N\}$ is a coordinate patch with $w_i=\frac{Z_i}{Z_0}$. The $Z_i$'s can be identified as generating sections in 	$H^0(\CP^N,\ocal(1))$. In particular, $Z_0$ is a local frame in $U_0$. Then on $U_0$, the Fubini-Study form $\omega=\frac{i}{2}\ddbar\log (1+|w|^2)$ has the following explicit form $$\omega=\frac{i}{2}\frac{(1+|w|^2)\sum dw^i\wedge d\bar{w}_i-(\sum \bar{w}_idw_i)(\sum w_id\bar{w}_i)}{(1+|w|^2)^2}$$
On each $U_{j,\alpha}$, within the area $0\leq t\leq t_0=\frac{\pi}{2\epsilon}+\log \epsilon$, the image under $\Phi_j$ is dominated by the two sections $s_{j,\alpha,0}$ and $s_{j,\alpha,1}$, since the contribution of other sections is of relative size $<k^2\epsilon$. We can estimate the map $[s_{j,\alpha,0},s_{j,\alpha,1},0,\cdots]$
by the local sections $[b_0,b_1z,0,\cdots]$, with relative error $<\frac{1}{k^2}$, where 
$b_0^{-2}=\frac{2^{k+2}\pi^{3/2}}{\sqrt{2k}}\epsilon^{-2k-1}$
 and $b_1^{-2}=\frac{2^{k+2}\pi^{3/2}}{\sqrt{2k}}e^{\pi/\epsilon}k^{2k+1}$. So the map is simplified to 
 	$$[1,e^{-\frac{\pi}{2\epsilon}}(\epsilon k)^{-k-1/2}z,0,\cdots]$$
 	So we can estimate the length of the image of $\gamma_{j,\alpha}$, which is approximately
 	$2\pi e^{-\frac{\pi}{2\epsilon}}(\epsilon k)^{-k-1/2}$, and goes to $0$ as $j\to \infty$. So the image of $\gamma_{j,\alpha}$ converges to $[1,0,\cdots]$ in the current ordering of coordinates. Similarly, we can estimate the length of the image of the circle $t_0=\frac{\pi}{2\epsilon}+\log \epsilon$, which is approximately $\frac{2\pi}{\epsilon}(\epsilon k)^{k+1/2}$, which goes to 0 as $j\to \infty$. Therefore, the image of the circle  $t_0=\frac{\pi}{2\epsilon}+\log \epsilon$ converges to $[0,1,\cdots]$ in the current ordering of coordinates. Notice that $e^{-\frac{\pi}{2\epsilon}}(\epsilon k)^{-k-1/2}|z|$ goes to $\infty$ at $t_0$, so the image of the area $0\leq t\leq t_0=\frac{\pi}{2\epsilon}+\log \epsilon$ converges to the $\CP^1$ connecting the points $[1,0,0,\cdots]$ and $[0,1,0,\cdots]$ in the current ordering of coordinates. This area is the bubble mentioned in the introduction. By symmetry, the image of the area $0\geq t\geq -t_0 $ also converges to a $\CP^1$.  
 	
 	\
 	
 	For the part $G_j^{-1}(C_0\backslash \cup U_{\alpha}(\epsilon(j,\alpha))) $, the sections $\{s_{j,\alpha,0}\}_{1\leq \alpha d}$ are negligible. So the convergence of the remaining sections in $V_j$ and $W_j$ in the $L^2$ sense implies that the image of $G_j^{-1}(C_0\backslash \cup U_{\alpha}(\epsilon(j,\alpha))) $ under $\Phi_j$ converges to the image of $\Phi_0$. To conclude the proof of the main theorem, we only have to notice that although $V_j\cup W_j$ is not orthonormal, we can modify them.
 	The sections in $W_j$ are almost orthonormal, so we can transform them to be orthonormal with a matrix $A_j$ whose difference from the identity matrix goes to $0$ as $j\to \infty$. And the modified sections still converge to sections in $W_0$ in $L^2$. We first apply a Gram-Schmidt process to the set $V_0$, then we apply a Gram-Schmidt process following the order of $V_0$ to the set $V_j\backslash \{s_{j,\alpha,0}\}_{1\leq \alpha d}$. Finally, since the sections $\{s_{j,\alpha,0}\}_{1\leq \alpha d}$ are almost orthonormal and almost orthogonal to the other section in $V_j$, we can modify the new $V_j$ again with a matrix $B_j$ whose difference from the identity matrix goes to $0$ as $j\to \infty$, to get a orthonormal set. And we have proved the main theorem.

	\bibliographystyle{plain}

	\bibliography{references}
	
\end{document}